\documentclass[a4paper,twoside,reqno]{amsart}

\usepackage[pagewise]{lineno}
\usepackage[utf8]{inputenc}

\usepackage{authblk}
\usepackage{hyperref}
\hypersetup{urlcolor=blue, citecolor=red}
\usepackage{amsmath,amsthm,amssymb,xcolor,bbm,mathrsfs,graphicx,subfigure,sidecap,wrapfig,float,caption}
\usepackage[english]{babel}
\usepackage{fancyhdr}
\newcommand{\n}{\hspace*{-6pt}}

\DeclareMathOperator{\unif}{uniform}

\newcommand{\ceil}[1]{\lceil #1\rceil}

\makeatletter
\@namedef{subjclassname@1991}{2020 Mathematics Subject Classification}
\makeatother
\subjclass{91C20, 91D25, 91D30, 94C15}
\keywords{Hegselmann-Krause model, opinion dynamics, probability of consensus, social network, critical value}

\title{A straightforward proof of the critical value in the Hegselmann-Krause model: up to one-half}
\author{Hsin-Lun Li}

\date{}
\email{hsinlunl@math.nsysu.edu.tw}

\theoremstyle{definition}
\newtheorem{theorem}{Theorem}

\newtheorem{lemma}[theorem]{Lemma}

\begin{document}

\allowdisplaybreaks

\thispagestyle{firstpage}
\maketitle
\begin{center}
    Hsin-Lun Li
    \centerline{$^1$National Sun Yat-sen University, Kaohsiung 804, Taiwan}
\end{center}
\medskip

\begin{abstract}
   In the Hegselmann-Krause model, an agent updates its opinion by averaging with others whose opinions differ by at most a given confidence threshold. With agents' initial opinions uniformly distributed on the unit interval, we provide a straightforward proof that establishes the critical value is up to one-half. This implies that the probability of consensus approaches one as the number of agents tends to infinity for confidence thresholds larger than or equal to one-half.
\end{abstract}

\section{Introduction}
The Hegselmann-Krause model comprises a finite set of agents, denoted as \( [n] = \{1, 2, \ldots, n\} \), each updating their opinions by averaging with opinion neighbors: the agents whose opinions differ by at most a confidence threshold \( \epsilon > 0 \). Interpreting in math, the update mechanism goes as follows:
$$x_i(t+1)=\frac{1}{|N_i(t)|}\sum_{j\in N_i(t)}x_j(t)$$ where
$$\begin{array}{rl}
     x_i(t)&\n=\hbox{opinion of agent $i$ at time $t$},\vspace{2pt} \\
     N_i(t)&\n=\{j\in [n]: |x_i(t)-x_j(t)|\leq \epsilon\}. 
\end{array}$$
In the synchronous Hegselmann-Krause model, every agent simultaneously updates their opinion at the next time step. In contrast, the asynchronous Hegselmann-Krause model involves only one uniformly selected agent updating its opinion at a time. Our analysis concentrates on the synchronous Hegselmann-Krause model within the unit interval \( [0,1] \), where the initial opinions of all agents are uniform random variables on $[0,1]$, expressed as \(x_i(0) = \unif([0,1])\) for all \(i \in [n]\). Denote \(x_{(i)}\) as the \(i\)th smallest number among \((x_j)_{j=1}^n\). An \emph{opinion graph} at time \(t\), represented as \(\mathscr{G}(t)\), is a graph with vertex set \( [n]\) and edge set \(E(t) = \{(i,j): |x_i(t) - x_j(t)| \leq \epsilon\}\). The notation \(\alpha \ll \beta\) signifies that \(|\alpha| \leq |\beta|\).

The Hegselmann-Krause model and the Deffuant model, as described in \cite{castellano2009statistical}, stand out as two popular models in opinion dynamics. In \cite{mHK}, the author introduced a mixed model capable of interpreting both the Hegselmann-Krause model and, as argued in \cite{mHK2}, the Deffuant model. The critical value of the Deffuant model is established as one-half in \cite{MR3069370}. Lower bounds for the probability of consensus in the Deffuant model and the Hegselmann-Krause model are respectively addressed in \cite{lanchier2020probability} and \cite{lanchier2022consensus}.

Unlike the Deffuant model, where agents interact in pairs, the Hegselmann-Krause model involves agents interacting in groups, constituting a more intricate interaction system. In our work, we provide a straightforward proof indicating that the critical value reaches up to one-half. Specifically, the probability of consensus tends toward one as the number of agents goes to infinity when the confidence threshold is larger than or equal to one-half.

\section{Main results}
\begin{theorem}\label{thm:critical value}
    The critical value in the Hegselmann-Krause Model is up to one-half.
\end{theorem}

\section{The model}
It has been established in \cite{mHK} that all agents' opinions converge in finite time within the synchronous Hegselmann-Krause model. We have identified several key properties of this model on the unit interval, including the order-preserving nature of all opinions. Additionally, we demonstrate that the opinion graph is almost surely connected initially and remains connected over time as $n \to \infty$ for a confidence threshold $\epsilon \geq 1/2$. This observation implies that the probability of consensus approaches one as $n \to \infty$ for $\epsilon \geq 1/2$.

\begin{lemma}[\cite{mHK}]\label{matching}
Given $\lambda_1,...,\lambda_n$ in $\mathbf{R}$ with $\sum_{i=1}^n\lambda_i=0$ and $x_1,...,x_n$ in $\mathbf{R^d}$. Then for $\lambda_1x_1+\lambda_2x_2+...+\lambda_nx_n,$ the terms with positive coefficients can be matched with the terms with negative coefficients in the sense that there are nonnegative values $c_i$ such that
\begin{align*}
     \sum_{i=1}^n\lambda_ix_i=\sum_{i,c_i\geq0,j,k\in[n]} c_{i}(x_{j}-x_{k})\text{ and }\sum_i c_i=\sum_{j,\lambda_j\geq0}\lambda_j.
\end{align*}
\end{lemma}
Lemma~\ref{matching} implies that the order of opinions among all agents persists over time.

\begin{lemma}[order-preserving]\label{order-preserving} 
If $x_i(t)\leq x_j(t)$, then $x_i(t+1)\leq x_j(t+1).$
\end{lemma}
\begin{proof}
Let $x=x(t),x'=x(t+1),$ and $N_i=N_i(t)$ for all $i\in[n].$ From Lemma \ref{matching},
\begin{align*}
    &x_j'-x_i'=(\frac{1}{|N_j|}-\frac{1}{|N_i|})\sum_{k\in N_i\cap N_j}x_k+\frac{1}{|N_j|}\sum_{k\in N_j-N_i}x_k-\frac{1}{|N_i|}\sum_{k\in N_i-N_j}x_k\\
    &=\left\{\begin{array}{lr}
         \displaystyle \sum_{k,p\in N_j-N_i,q\in N_i}a_k(x_p-x_q)&\hbox{if}\ |N_j|\geq |N_i|  \\
         \displaystyle \sum_{k,p\in N_j,q\in N_i-N_j}a_k(x_p-x_q)&\hbox{else,}
    \end{array}\right.
\end{align*}
 where $a_k\geq0$ for all $k.$ We claim that
 \begin{enumerate}
     \item If $q\in N_i$ and $p\in N_j-N_i$, then $x_q<x_p$.\label{v1}\vspace{2pt}
     \item If $q\in N_i-N_j$ and $p\in N_j$, then $x_q<x_p$.\label{v2}
 \end{enumerate}
 \begin{proof}[Proof of Claim \ref{v1}]
 Assume by contradiction that there exist $q\in N_i$ and $p\in N_j-N_i$ such that $x_q\geq x_p$. Then, $x_j-\epsilon\leq x_p< x_i-\epsilon$, a contradiction.
 \end{proof} 
 \begin{proof}[Proof of Claim \ref{v2}]
  Assume by contradiction that there exist $q\in N_i-N_j$ and $p\in N_j$ such that $x_q\geq x_p$. Then, $x_i+\epsilon\geq x_q> x_j+\epsilon$, a contradiction.
 \end{proof}
 Either way, $x_j'-x_i'\geq0.$ This completes the proof.
\end{proof}
Lemma~\ref{order-preserving} reveals that when the opinion graph is disconnected, it remains so over time.

\begin{lemma}[disconnected-preserving]\label{disconnected-preserving} 
If $\mathscr{G}(t)$ is disconnected, then $\mathscr{G}(t+1)$ remains disconnected.
\end{lemma}
\begin{proof}
Assume $x_1(t)\leq x_2(t)\leq...\leq x_n(t).$ Since $\mathscr{G}(t)$ is disconnected, $$x_{i+1}(t)-x_i(t)>\epsilon\text{ for some }i\in[n-1].$$ Since vertices $i$ and $i+1$ have respectively no neighbors on their right and left at time $t$, $$x_i(t+1)\leq x_i(t)\text{ and }x_{i+1}(t)\leq x_{i+1}(t+1).$$ Hence $x_{i+1}(t+1)-x_i(t+1)>\epsilon.$ From Lemma \ref{order-preserving}, $$x_1(t+1)\leq x_2(t+1)\leq...\leq x_n(t+1).$$ Thus $\mathscr{G}(t+1)$ is disconnected.
\end{proof}

\begin{lemma}
    A consensus cannot be achieved if the initial opinion graph is disconnected.
\end{lemma}

\begin{lemma}\label{equvalence relation for consensus}
    A consensus can be achieved if and only if the opinion graph remains connected over time.
\end{lemma}

\begin{lemma}\label{equivalent relation for connectedness}
    An opinion graph is connected at time $t$ if and only if edge $((i),(i+1))\in E(t)$ for all $i\in [n-1].$
\end{lemma}

\begin{lemma}\label{condition to keep edge}
    Edge $(i,j)\in E(t+1)$ if  $(i,j)\in E(t)$  with $$2\max\{|N_{i}(t)-N_{j}(t)|,\ |N_{j}(t)-N_{i}(t)|\}\leq |N_{i}(t)\cap N_{j}(t)|.$$
\end{lemma}

\begin{proof}
    Without loss of generality, we claim that $(1,2)\in E(t)$ implies $(1,2)\in E(t+1)$ for all $t\geq 0.$ Let $x_i=x_i(t)$, $x_i^\star=x_i(t+1)$ and $N_i=N_i(t)$ for $i=1,2$ and let $n_1=|N_1-N_2|$, $n_2=|N_1\cap N_{2}|$ and $n_3=|N_{2}-N_1|$. Say $(a_i)=(x_i)_{i\in N_1-N_2}$, $(b_i)=(x_i)_{i\in N_1\cap N_2}$ and $(c_i)=(x_i)_{i\in N_2-N_1}.$ Then, 
    \begin{align}
         x_1^\star-x_{2}^\star&=\frac{\sum_{i}a_i+\sum_{i}b_i}{n_1+n_2}-\frac{\sum_{i}b_i+\sum_{i}c_i}{n_2+n_3}\notag\\
        &=\frac{(n_2+n_3)\sum_i a_i+(n_3-n_1)\sum_i b_i-(n_1+n_2)\sum_i c_i}{(n_1+n_2)(n_2+n_3)}.
    \end{align}
    By symmetry, assuming that $n_3\geq n_1$, we match $(n_2+n_3)n_1$ terms of form $a_i$ and $(n_3-n_1)n_2$ terms of form $b_i$ with $(n_1+n_2)n_3=(n_2+n_3)n_1+(n_3-n_1)n_2$ terms of form $c_i$, therefore by the triangle inequality,
    \begin{align*}
        &(n_2+n_3)\sum_i a_i+(n_3-n_1)\sum_i b_i-(n_1+n_2)\sum_i c_i=\sum_i(\hat{a}_i-\hat{c}_i)+\sum_i(\tilde{b}_i-\tilde{c}_i)\\
        &\hspace{1cm}\ll (n_2+n_3)n_1 3\epsilon+(n_3-n_1)n_2 2 \epsilon=\epsilon(n_1n_2+3n_1n_3+2n_2n_3).
    \end{align*}
    Observe that 
    \begin{equation*}
        n_1n_2+3n_1n_3+2n_2n_3\leq (n_1+n_2)(n_2+n_3)\iff 2n_1n_3+n_2n_3\leq n_2^2.
    \end{equation*}
    It follows from the assumption that $2 n_3\leq n_2$, therefore $$2n_1n_3+n_2n_3\leq 2(n_2/2)^2+n_2(n_2/2)=n_2^2,\ \hbox{which implies}\ x_1^\star-x_{2}^\star\ll \epsilon.$$
    Thus, $(1,2)\in E(t+1).$
\end{proof}

\begin{lemma}\label{connected initially}
    The initial opinion is almost surely connected as $n\to \infty.$
\end{lemma}
\begin{proof}
    Observe that $P(\mathscr{G}(0)\ \hbox{disconnected})\leq (1-\epsilon)^{(n-2)}\to 0$ as $n\to \infty,$ which implies $P(\mathscr{G}(0)\ \hbox{disconnected})\to 0$ as $n\to \infty.$ 
\end{proof}

\begin{lemma}\label{condition to preserve connectedness}
    Given confidence threshold $\epsilon\geq 1/2$. Let $\mathscr{H}_t$ be the following statement: $((1),(\ceil{n/2}))$, $ ((\ceil{n/2}),(n))\in E(t)$ with
    $$\begin{array}{rcl}
      2\max\{|N_{(1)}(t)-N_{(\ceil{n/2})}(t)|,\ |N_{(\ceil{n/2})}(t)-N_{(1)}(t)|\}&\n\leq\n& |N_{(1)}(t)\cap N_{(\ceil{n/2})}(t)|\vspace{2pt} \\
         2\max\{|N_{(\ceil{n/2})}(t)-N_{(n)}(t)|,\ |N_{(n)}(t)-N_{(\ceil{n/2})}(t)|\}&\n\leq\n& |N_{(\ceil{n/2})}(t)\cap N_{(n)}(t)|
    \end{array}$$ 
    Then as $n\to\infty$, $\mathscr{H}_t$ holds for all $t\geq 0.$
\end{lemma}

\begin{proof}

   Given confidence threshold $\epsilon\geq 1/2$. We claim that edge $(i,j)\in E(0)$ implies $|N_i(0)\cap N_j(0)|=X+2\to (n-2)p+2\geq n/2+1$ for $n$ large. Without loss of generality, let edge $(1,2)\in E(0)$ and $X=\sum_{k=3}^n\mathbbm{1}\{x_k(0)\in N_1(0)\cap N_2(0)\}$. Then,  $\mathbbm{1}\{x_k(0)\in N_1(0)\cap N_2(0)\}=\rm{Bernoulli}(p)$ with $p\geq \epsilon$ for all $k=3,\ldots,n.$ By the strong law of large numbers, $X/(n-2)\to p$ almost surely as $n\to \infty$ so $|N_1(0)\cap N_2(0)|=X+2\to (n-2)p+2\geq n/2-1+2=n/2+1$ for $n$ large. Hence, by Lemmas~\ref{connected initially} and \ref{equivalent relation for connectedness}, $\mathscr{H}_0$ holds. For $t\geq 1$, by Lemmas~\ref{condition to keep edge} and \ref{order-preserving}, $\mathscr{H}_{t-1}$ implies $\mathscr{H}_t.$ By induction, we are done.
    
\end{proof}

\begin{proof}[\bf Proof of Theorem~\ref{thm:critical value}]
    By Lemmas~\ref{equvalence relation for consensus} and \ref{condition to preserve connectedness}, a consensus can almost surely be achieved as $n\to \infty$ for confidence threshold $\epsilon\geq 1/2$.
\end{proof}

\section{Statements and Declarations}
\subsection{Competing Interests}
The author is supported by NSTC grant.

\subsection{Data availability}
No associated data was used.

\end{document}